\documentclass[12pt,reqno]{amsart}

\usepackage{fix-cm} 
\usepackage{microtype}
\usepackage[T1]{fontenc}
\usepackage{amsthm, amssymb, bm}
\usepackage{stmaryrd} 
\usepackage{fix-cm}

\usepackage[a4paper, left=2.5cm, right=2.5cm, top=45mm, bottom=20mm, includehead]{geometry}

\usepackage{enumerate}

\usepackage{marginnote}
\usepackage{float}

\usepackage{newtxmath} 

\usepackage[dvipsnames]{xcolor} 

\usepackage{hyperref}
\usepackage{cleveref}
\hypersetup{
	linkcolor  = violet,
	citecolor  = YellowOrange,
	urlcolor   = Aquamarine,
	colorlinks = true,
}
 
\RequirePackage{multicol} 
\RequirePackage{multirow}

\newtheorem{theorem}{Theorem}[section]
\newtheorem{theorem*}{Theorem}
\newtheorem{corollary}[theorem]{Corollary}
\newtheorem{lemma}[theorem]{Lemma}
\newtheorem{proposition}[theorem]{Proposition}
\newtheorem{definition}[theorem]{Definition}
\newtheorem{example}[theorem]{Example}

\numberwithin{equation}{section}

\usepackage{xcolor}

\newcommand{\fini}{{\operatorname{end}}}
\newcommand{\fin}[1]{\fini({#1})}
\newcommand{\com}[1]{\operatorname{beg}({#1})}

 \title[Gaps in binary cyclotomic polynomials]{Gaps in binary cyclotomic polynomials}

\begin{document}

\markboth{A. Cafure, E. Cesaratto}{Gaps in binary cyclotomic polynomials}

\title{Gaps in binary cyclotomic polynomials}

\author{Antonio Cafure}

\address{A. Cafure\\ Instituto del Desarrollo Humano\\ Universidad Nacional de General Sarmiento $\&$ CONICET\\
Los Polvorines, Buenos Aires, Argentina\\
\newline{acafure@campus.ungs.edu.ar} }

\author{Eda Cesaratto}

\address{E. Cesaratto\\ Instituto del Desarrollo Humano\\ Universidad Nacional de General Sarmiento $\&$ CONICET\\
Los Polvorines, Buenos Aires, Argentina
\\
\newline{ecesaratto@campus.ungs.edu.ar} }

\begin{abstract}
For odd prime numbers $p < q$, let $\Phi_{pq} \in \mathbb{Z}[X]$ be the binary cyclotomic polynomial of order $pq$. In this paper, we prove that  the second gap of  $\Phi_{pq}$ is the maximum of $r-1$ and $p-r-1$, where $r$ is the remainder of $q$ divided by $p$. For $q$  congruent to $\pm 1$ modulo $p$, we determine the number of gaps for each possible length. To obtain these results, we develop a new approach in which the coefficients of $\Phi_{pq}$ are described as concatenations of words arising from iterations of a circular map.
\end{abstract}

\maketitle




\section{Introduction}

For an integer $n\ge 1$, we denote by $\Phi_{n}(X)$ the \emph{cyclotomic polynomial} of order $n$, that is, the monic polynomial in $\mathbb{Z}[X]$ of degree $\varphi(n)$ whose roots are the primitive roots of unity of order $n$ (here $\varphi$ is Euler's totient function). The goal of this paper is to study the set of gaps of \emph{binary cyclotomic polynomials} $\Phi_{pq}$ with $p < q$ prime numbers. The study of gaps of cyclotomic polynomials started with an application to cryptography in~\cite{Hong12} and now it is a problem of independent interest (see the survey~\cite{Sanna22}). 

The \emph{gapset} of a polynomial $\Phi=c_{n}X^{e_{n}} +  c_{n-1}X^{e_{n-1}} + \cdots  +  c_{0}X^{e_{0}}$ with   nonzero coefficients $c_{n}, c_{n-1}, \ldots, c_{0}$  and strictly decreasing exponents $e_{n}>e_{n-1} > \cdots > e_{0}$ is defined as 
\[G(\Phi)=\{e_{i+1}-e_{i} : 0\le i\le n-1\},\]
the set of  differences between the exponents of its consecutive nonzero monomials. We refer to each element of $G(\Phi)$ different from 1 as a \emph{gap} or as the \emph{length of a gap}.

The gapset $G(\Phi_{p})$ of the cyclotomic polynomial $\Phi_{p}=X^{p-1}+X^{p-2}+\cdots + X+1$ of prime order $p$ is $ \{1\}$. From the well-known identity $\Phi_{2q} (X) = \Phi_{q}(-X)$ for an odd prime $q$, it is also immediate that $G(\Phi_{2q}) = \{1\}$. 

Thus, the first nontrivial case is the gapset $G(\Phi_{pq})$ for distinct odd prime numbers $p<q$. The maximum 
of $G(\Phi_{pq})$, 
\begin{align*} g_{1}(\Phi_{pq})&=\max G(\Phi_{pq}),
\end{align*}
 is called the \emph{maximum} or \emph{first} gap of the polynomial $\Phi_{pq}$. 
 
 It is known that $g_{1}(\Phi_{pq}) = p-1$ and that there are precisely $2\lfloor q/p\rfloor$ such gaps. Different proofs of these results can be found in~\cite{Hong12},~\cite{Moree14},~\cite{Camburu16}, and~\cite{Zhang16}. 

In the particular case of $q$ congruent to $\pm 1$ modulo $p$, Camburu et al.\ proved in \cite[Theorem 3.2 (ii)]{Camburu16} that $\Phi_{pq}$ has gaps of length $i$ for each $2\le i \le p-1$.

Beyond the binary case, the results established in~\cite{AlKateeb20} state that if $m$ is squarefree, $q$ is prime, and $m < q$, then $g_{1}(\Phi_{mq}) = \varphi(m)$.
A summary of the results stated above can be found in the survey article~\cite{Sanna22} which presents a state-of-the-art of what is known about the dense representation of cyclotomic polynomials.

To the best of our knowledge, there are no previous results concerning smaller gaps of $\Phi_{pq}$. In this article we present results concerning what we call the second gap of binary cyclotomic polynomials.
If $G(\Phi_{pq})\setminus \{ g_{1}(\Phi_{pq})\}$ has elements distinct from $1$, the \emph{second gap} of $\Phi_{pq}$ is the number
\begin{align*}  g_{2}(\Phi_{pq})&=\max( G(\Phi_{pq})\setminus \{ g_{1}(\Phi_{pq})\}). 
\end{align*}

Our Theorem~\ref{thm:teorema 1} provides the exact value of $g_2(\Phi_{pq})$ and a lower bound for the number of such gaps. 
\begin{theorem}\label{thm:teorema 1} Let $5\le p < q$ be prime numbers, let $\lfloor q/p\rfloor$ be  the quotient, and let $r$ be the remainder of the division of $q$ by $p$. The second gap $g_2(\Phi_{pq})$ is equal to $\max\{r-1,p-r-1\}$. 
Moreover, there are at least $2\lfloor q/p\rfloor$ such gaps.  
\end{theorem}

In this paper, we also expand the results of \cite[Theorem 3.2 (ii)]{Camburu16}, which deals with $q \equiv \pm 1
\!\!\!\mod \!p$, by determining the number of gaps of  length $i$ for each $2\le i\le p-1$. 

\begin{theorem}\label{thm:teorema 2} Let $ p < q$ be odd prime numbers, let $\lfloor q/p\rfloor$ be the quotient, and let $r$ be the remainder of the division of $ q$ by $p$. 

\begin{enumerate}
\item  If $p=3$, then $\Phi_{3q}$ has exactly $2\lfloor q/3\rfloor$ gaps of length $2$. 

\smallskip

\item If $p\ge 5$ and $r=1$, then
 $\Phi_{pq}$ has exactly $2\lfloor q/p\rfloor$ gaps of length $i$ for each $2\le i\le p-1$. 

\smallskip

 \item If $p\ge 5$ and $r=p-1$, then $\Phi_{pq}$ has exactly $2\lfloor q/p\rfloor$ gaps of length $p-1$ and $2\lfloor q/p\rfloor +2$ gaps of length $i$ for each $2\le i\le p-2$. 
 \end{enumerate}
\end{theorem} 

We remark that these results also remain valid for the so-called semigroup polynomials when $p$ and $q$ are relatively prime numbers, but not necessarily prime (see \cite{Moree14}).

To obtain our results we develop a novel approach that enables the efficient location of nonzero coefficients. 
 We appeal to our representation theorem given in \cite[Theorem 1]{CaCe21}, restated here as Theorem~\ref{thm:representation theorem}. Our approach complements those of Lenstra \cite{Lenstra79} and Lam–Leung \cite{LaLe99} (see also \cite{Moree14}). 
Their explicit formulas for the coefficients of $\Phi_{pq}$ are primarily arithmetic and focus on individual coefficients. Our word-based approach emphasizes the combinatorial structure, and is therefore better suited for the study of gaps.

 In our framework, the vector of coefficients of $\Phi_{pq}$ is represented as a word over the ternary alphabet $\{-1, 0, 1\}$ and is expressed compactly as concatenations and fractional powers of $p-1$ words $\boldsymbol{\omega}_i$ of length $p$, namely the \emph{basic words}, rather than as a sequence of $(p-1)(q-1) + 1$ individual coefficients of $\Phi_{pq}$. 
 Furthermore, there is a close connection between the index $i$ and the length of the gaps. In fact, we prove that the word $\boldsymbol{\omega}_0$ contains the largest gap, the word $\boldsymbol{\omega}_1$ contains the second largest gap, and for $i \ge 2$, every gap in $\boldsymbol{\omega}_{i}$ is at most as large as the second largest gap.
 
 Since these basic words are partial sums of the orbit of an initial word (of length $p$) under a circular permutation, we study the interplay between dynamics and arithmetic which allows us to control both the length of the gaps and their positions.
 First, we control the positions of nonzero coefficients within basic words. Then, in a second step, we show that concatenating these basic words preserves the number and length of the gaps, with the exception of a few cases that are handled separately. 
 
 In addition, we also use two elementary facts: cyclotomic polynomials are reciprocal (see, for instance,~\cite{Niven56} or~\cite{Cafure21} for a new proof), and the middle coefficient of $\Phi_{pq}$ is nonzero (see, for instance, Theorem 2.1 in the survey~\cite{Sanna22}).

\subsection*{Structure of the paper}
In Section~\ref{section:our results}, we introduce the terminology on words used throughout the article. In particular, we define the notion of a gap of a word, extending the notion of a gap of a polynomial. We also state our representation theorem (\cite[Theorem 1]{CaCe21}). 

In Section~\ref{section:the second gap} we state Theorem~\ref{thm:second gap}, which provides the locations of the maximum gaps within the vector of coefficients of $\Phi_{pq}$, expanding \cite[Theorem 3.2 (i)]{Camburu16}, and the locations of the first $2\lfloor q/p\rfloor$ second gaps. We give a proof for $p \geq 5$ and $2\leq r \leq p-2$. 

In Section~\ref{section:gaps and a representation theorem} we deal with 
 the cases $r = 1$ and $r = p-1$ (which include $p =3$), completing the proof of Theorem~\ref{thm:second gap}. The main result of the section is Theorem~\ref{thm:descripcion word pm1}, a representation theorem for the vector of coefficients of $\Phi_{pq}$ for $q \equiv \pm 1
\!\!\!\mod \!p$, well-adapted to the study of gaps. For $p=3$, Theorem~\ref{thm:descripcion word pm1} provides a word-based formulation of a theorem by Habermehl, Richardson and Szwajkos on the coefficients of $\Phi_{3q}$, published in~\cite{Habermehl64}. 

Theorem~\ref{thm:teorema 1} is a consequence of Theorem~\ref{thm:second gap}. 
Theorem~\ref{thm:teorema 2} follows from Theorem~\ref{thm:descripcion word pm1} and its proof is given in Section~\ref{subsec:completing proof of theorem 2}. 

\section{A representation theorem} \label{section:our results}
\subsection{Notations}\label{sect: results-notations} We use standard notations from combinatorics on words (see, for instance, \cite{Lothaire97}). 
Given a word $\boldsymbol{u}$ over an alphabet $\mathcal{A}$ (a finite nonempty set), we denote by $|\boldsymbol{u}|$ its length. The set $\mathcal{A}^{n}$ consists of all words of length $n$. 
We write $\mathcal{A}^{*} = \bigcup_{n \in \mathbb{Z}_{\ge 0}} \mathcal{A}^{n}$ and $\mathcal{A}^{+} = \mathcal{A}^{*}\setminus \{\epsilon\}$, where $\epsilon$ is the empty word. 

We write $\boldsymbol{u}[j]$ for the $j$-th letter of $\boldsymbol{u}$ and $\boldsymbol{u}[j:t]$ for the subword of $\boldsymbol{u}$ beginning at position~$j$ and ending at position~$t$ for $0 \le j \le t < |\boldsymbol{u}|$.

The concatenation of two words $\boldsymbol{u}$ and $\boldsymbol{v}$ is written as $\boldsymbol{u}\boldsymbol{v}$ and, for $s \in \mathbb{Z}_{\ge 0}$, the {$s$-power} of $\boldsymbol{u}$ is written as $\boldsymbol{u}^{s}$. 
For positive integers $k$ and $n$, the remainder of the division of $k$ by $n$ is denoted by $k \!\!\! \mod \! n$. 
For $\boldsymbol{u}\in \mathcal{A}^{n}$, the \emph{fractional power} $\boldsymbol{u}^{k/n}$ of $\boldsymbol{u}$ is the element of $\mathcal{A}^{k}$ defined as follows:
\begin{itemize}
\item If $k < n$, then $\boldsymbol{u}^{k/n}$ is the word $\boldsymbol{u}[0:k-1] = u_{0}\cdots u_{k-1}$.

\medskip

\item If $k \ge n$, then 
$
\boldsymbol{u}^{k/n} = \boldsymbol{u}^{\lfloor k/n \rfloor}\boldsymbol{u}^{(k \!\!\mod \! n)/n}$.
\end{itemize}

From now on, the alphabet is $\mathcal{A} =\{-1,0,1\}$. Following the ideas of ~\cite{Camburu16}, we first introduce the notion of gapblock of a word. 
Any $\boldsymbol{u}\in \mathcal{A}^{+}$ having at least a zero letter can be uniquely decomposed as a concatenation of blocks of zeros and blocks of nonzero symbols as follows: 
\begin{equation}\label{eq:word decomposition}\boldsymbol{u}=\boldsymbol{u}_{0}0^{k_1}\boldsymbol{u}_1 0^{k_2}\boldsymbol{u}_2\cdots \boldsymbol{u}_{t-1}0^{k_t} \boldsymbol{u}_t. 
\end{equation}
Here $t, k_{1},\dots,k_{t}\in \mathbb Z_{\ge 1}$, the words $ \boldsymbol{u}_{0}, \boldsymbol{u}_{t} \in \{-1,1\}^{*}$, and the words $ \boldsymbol{u}_1,\dots,\boldsymbol{u}_{t-1} \in \{-1,1\}^{+}$. 

The words $0^{k_1}, \dots, 0^{k_t}$ in \eqref{eq:word decomposition} are called the \emph{gapblocks} of $\boldsymbol{u}$ and the numbers $k_1+1,\dots, k_t+1$ are called the \emph{gaps} of
$\boldsymbol{u}$. For each $\boldsymbol{u}\in \mathcal{A}^{+}$ having at least a zero letter, 
the \emph{gapset} of $\boldsymbol{u}$ is \[G(\boldsymbol{u})=\{k_{i} +1, 1\le i\le t\}.\]
The \emph{maximum or first gap} and the \emph{second gap} of $\boldsymbol{u}$ are defined as 
\[g_1(\boldsymbol{u})=\max(\boldsymbol{u})\quad \text{ and }\quad g_2(\boldsymbol{u})=\max \left(G(\boldsymbol{u})\setminus\{g_1(\boldsymbol{u})\}\right).\]

A word $\boldsymbol{u}\in \{-1,1\}^{*}$ does not have gapblocks. The gaps are defined by adding $1$ to the lengths of the gapblocks, so as to coincide with the notion of a gap of a polynomial.

\begin{example}\label{example: decomposition of a word} {\rm Let $\boldsymbol{u}$ be the word 
\[\boldsymbol{u} = 1(-1)000001(-1)000001(-1)01(-1)001(-1)0 1(-1)001(-1)0 1(-1)0 10 . \] 
The decomposition of $\boldsymbol{u}$ according to \eqref{eq:word decomposition} is as follows:
\begin{align*}\boldsymbol{u} =&\underbrace{1(-1)}_{\boldsymbol{u}_{0}}0^{5} \underbrace{1(-1)}_{\boldsymbol{u}_{1}}0^5\underbrace{1(-1)}_{\boldsymbol{u}_{2} }0\underbrace{1(-1)}_{\boldsymbol{u}_{3}}0^2 \underbrace{1(-1)}_{\boldsymbol{u}_{4}}0\underbrace{1(-1)}_{\boldsymbol{u}_{5}}0^2\underbrace{1(-1)}_{\boldsymbol{u}_{6}}0 
\underbrace{1(-1)}_{\boldsymbol{u}_{7}}0\underbrace{1}_{\boldsymbol{u}_{8}} 0 \underbrace{\epsilon}_{\boldsymbol{u}_{9}}.
 \end{align*}
 The distinct gapblocks of $\boldsymbol{u}$ are $0^5, 0^2, 0$ of length $5,2,1$, respectively. Thus, $G (\boldsymbol{u}) = \{6,3,2\}$. }
\end{example}
\subsection{The representation theorem}
We begin by defining our main object of study, the \emph{binary cyclotomic word} $\boldsymbol{a}_{pq}$. 
 
\begin{definition} Let $\mathcal{A}=\{-1,0,1\}$, let $p < q$ be prime numbers, and let $m = \varphi(pq) = (p-1)(q-1)$. The \emph{binary cyclotomic word} $\boldsymbol{a}_{pq}\in \mathcal{A}^{m + 1}$ is the word whose letters are the coefficients of the binary cyclotomic polynomial $\Phi_{pq}$. 
\end{definition}
Theorem~\ref{thm:representation theorem} below, Theorem~1 in our article~\cite{CaCe21}, is what we refer to as our representation theorem. It describes the structure of $\boldsymbol{a}_{pq}$ as a concatenation of $p-1$ words of length $p$ and its fractional powers. Its proof is based on a transfer between algebra and combinatorics on words. We now introduce the remaining definitions. 

 The mapping $\sigma: \mathcal{A}^{p} \rightarrow \mathcal{A}^{p}$ defined by 
 \begin{align*} 
 \sigma( u_{0} u_{1} \cdots u_{p-2} u_{p-1})
 = u_{1} \cdots u_{p-2} u_{p-1}u_{0}
 \end{align*}
is called the \emph{left circular permutation}. For $s\in \mathbb{Z}_{\geq 0}$, the $s$-th iterate of $\sigma$ is denoted by $\sigma^{s}$.
 
Starting with $\boldsymbol{d}_{0} = 1(-1)0 \cdots 0 \in \mathcal{A}^{p}$, we recursively define the following words in $\mathcal{A}^{p}$:
\begin{equation}
\label{defi:definition of di}
\boldsymbol{d}_i = \sigma^{q}( \boldsymbol{d}_{i-1}) = \sigma^{iq}( \boldsymbol{d}_0), \quad \text{for } i \in \llbracket 1, p-1\rrbracket.
\end{equation} 
Here the symbol $\llbracket{a},{b}\rrbracket$ denotes the set $\{k \in \mathbb{Z}:a \leq k \leq b\}$.

Finally, we introduce an addition operation between words of equal length analogous to the usual addition of vectors. Given words $\boldsymbol{u} = u_{0} u_{ 1}\cdots u_{k-1}$ and $\boldsymbol{v} = v_{0}v_{1}\cdots v_{k-1}$ over $\mathbb{Z}$, their sum is the word over $\mathbb{Z}$ defined by $\boldsymbol{u} + \boldsymbol{v} = (u_{0} + v_{0})(u_{1} + v_{1} )\cdots (u_{k-1} + v_{k-1})$.

We now define the following $p$ words in $\mathcal{A}^{p}$, which we call \emph{basic words}:
\begin{equation}\label{eq:definition of omega_i}
\boldsymbol{\omega}_0 = \boldsymbol{d}_0, \qquad \boldsymbol{\omega}_i = \boldsymbol{\omega}_{i-1} + \boldsymbol{d}_i, \quad \text{for } i \in \llbracket 1, p-1\rrbracket.
\end{equation}
By \cite[Lemma 1]{CaCe21}, it follows that $\boldsymbol{\omega}_{i}$ belongs to $\mathcal{A}^{p}$ for each $i \in \llbracket 1, p-1\rrbracket$. 

\begin{theorem}\label{thm:representation theorem}
 Let $p < q$ be prime numbers and let $r$ denote the remainder of the division of $q$ by $p$. The following assertions hold: 
\begin{enumerate}
\item\label{item thm1 a} The words $\boldsymbol{\omega}_0, \ldots, \boldsymbol{\omega}_{p-2}$ depend only on the pair $(p,r)$. 
\item The binary cyclotomic word 
$\boldsymbol{a}_{pq}$ 
 is obtained as the concatenation
 \[
 \boldsymbol{a}_{pq}= \boldsymbol{\omega}_{0}^{q/p}
\boldsymbol{\omega}_{1}^{q/p} \cdots \,\boldsymbol{\omega}_{p-3}^{q/p}\,\, \boldsymbol{\omega}_{p-2}^{(q-p+2)/p}.
 \]
\end{enumerate}
\end{theorem}

In the next example we use Theorem~\ref{thm:representation theorem} to obtain the decomposition of the binary cyclotomic word $\boldsymbol{a}_{7 \cdot 17}$. First we provide the dense representation of $\Phi_{7 \cdot 17}$:
\medskip
\begin{align*}\Phi_{7\cdot 17}&=X^{96} - X^{95} + X^{89} - X^{88} + X^{82} - X^{81} + X^{79} - X^{78} + X^{75} - X^{74} + X^{72}\\ 
&- X^{71} + X^{68} - X^{67} + X^{65} - X^{64} + X^{62} - X^{60}+ X^{58} - X^{57} + X^{55} - X^{53} \\ 
& + X^{51} - X^{50} + X^{48} - X^{46} + X^{45} - X^{43} + X^{41} - X^{39} + X^{38} - X^{36} + X^{34}\\& - X^{32} + X^{31} - X^{29} + X^{28} - X^{25} + X^{24} - X^{22} + X^{21} - X^{18}+ X^{17} - X^{15} 
\\
&+ X^{14} - X^{8} + X^{7} - X + 1.
\end{align*}

\begin{example}\label{example:first example}{\rm 
 Let $p = 7 $ and let $q > 7$ be  any prime such that  $r=3$. 
 
 Starting with $\boldsymbol{d}_{0} = 1(-1)00000$, we compute the words
$\boldsymbol{d}_{i} = \sigma^{q}(\boldsymbol{d}_{i-1}) =  \sigma^{3} (\boldsymbol{d}_{i-1})$  for  $i \in \llbracket 1,  5\rrbracket$:  
\begin{align*}
 \boldsymbol{d}_{0} & = 1(-1)00000  &      \boldsymbol{d}_{1} & = 00001(-1)0 & \boldsymbol{d}_{2} & = 01(-1)0000
 \\
 \boldsymbol{d}_{3} & = 000001(-1) & \boldsymbol{d}_{4} & = 001(-1)000  &      \boldsymbol{d}_{5} & =(-1)000001.  
 \end{align*}
 Next, we compute the basic words: 
 \begin{align*}
 \boldsymbol{\omega}_{0} & = 1(-1)00000 & \boldsymbol{\omega}_{1} & = 1(-1)001(-1)0 & \boldsymbol{\omega}_{2} & = 10(-1)01(-1)0 
 \\
\boldsymbol{\omega}_{3} & = 10(-1)010(-1) & \boldsymbol{\omega}_{4} & = 100(-1)10(-1) & \boldsymbol{\omega}_{5} & = 000(-1)100.
 \end{align*}
The fractional powers  $\boldsymbol{\omega}_{i}^{3/7}$ are given by
 \begin{align*}
 \boldsymbol{\omega}_{0}^{3/7}& = 1(-1)0 & \boldsymbol{\omega}_{1}^{3/7} & = 1(-1)0 & \boldsymbol{\omega}_{2}^{3/7} & = 10(-1)
 \\
\boldsymbol{\omega}_{3}^{3/7} & = 10(-1)  & \boldsymbol{\omega}_{4}^{3/7} & = 100  & \boldsymbol{\omega}_{5}^{3/7} & = 000 .
 \end{align*} 
Since   $\varphi(7 q) = 6(q-1)$,  Theorem~\ref{thm:representation theorem} implies that $\boldsymbol{a}_{7 q}$ is the following word in $\mathcal{A}^{6(q-1) + 1}$:
\begin{align*}
  \boldsymbol{a}_{7 q} 
  & = \boldsymbol{\omega}_{0}^{q/7}   \boldsymbol{\omega}_{1}^{q/7} \boldsymbol{\omega}_{2}^{q/7} \boldsymbol{\omega}_{3}^{q/7}\boldsymbol{\omega}_{4}^{q/7}\boldsymbol{\omega}_{5}^{(q-5)/7} .
  \end{align*}
  
In particular, the  cyclotomic word $\boldsymbol{a}_{7\cdot 17}$ is given by 
\begin{align*}\boldsymbol{a}_{7\cdot 17}&=\underbrace{1(-1)00000}_{\boldsymbol{\omega}_{0}} \underbrace{1(-1)00000}_{\boldsymbol{\omega}_{0}}\underbrace{1(-1)0}_{\boldsymbol{\omega}_{0}^{3/7} } \underbrace{1(-1)001(-1)0}_{\boldsymbol{\omega}_{1}} \underbrace{1(-1)001(-1)0}_{\boldsymbol{\omega}_{1}}\underbrace{1(-1)0}_{\boldsymbol{\omega}_{1}^{3/7} } \\
&\underbrace{10(-1)01(-1)0}_{\boldsymbol{\omega}_{2}} \underbrace{10(-1)01(-1)0}_{\boldsymbol{\omega}_{2}}\underbrace{\boldsymbol{1}0(-1)}_{\boldsymbol{\omega}_{2}^{3/7} } \underbrace{10(-1)010(-1)}_{\boldsymbol{\omega}_{3}} \underbrace{10(-1)010(-1)}_{\boldsymbol{\omega}_{3}}\underbrace{10(-1)}_{\boldsymbol{\omega}_{3}^{3/7} }\\
&\underbrace{100(-1)10(-1)}_{\boldsymbol{\omega}_{4}} \underbrace{100(-1)10(-1)}_{\boldsymbol{\omega}_{4}}
\underbrace{100}_{\boldsymbol{\omega}_{4}^{3/7} } \underbrace{000(-1)100}_{\boldsymbol{\omega}_{5}} \underbrace{000(-1)1 }_{\boldsymbol{\omega}_{5}^{5/7}  }.
\end{align*}

The word $\boldsymbol{a}_{7\cdot 17}$ is reciprocal, i.e., it is equal to its reverse. This implies that  the word $\boldsymbol{a}_{7\cdot 17}$ may be written as 
\[\boldsymbol{a}_{7\cdot 17}=\boldsymbol{c}_{7\cdot 17}\,\boldsymbol{1}\, \widetilde{\boldsymbol{c}_{7\cdot 17}},\]
where    $\widetilde{\boldsymbol{c}_{7\cdot 17}}$ is the word obtained by reversing the letters of $\boldsymbol{c}_{7\cdot 17}$. Notice   that $G(\boldsymbol{a}_{7\cdot 17}) = G(\boldsymbol{c}_{7\cdot 17})$. 

We have that  $g_{1}(\boldsymbol{c}_{7\cdot 17})=6$, and it occurs in each $\boldsymbol{\omega}_0$, and that  $g_{2}(\boldsymbol{c}_{7\cdot 17}) = 3$, and it occurs in $\boldsymbol{\omega}_1$. 
 Therefore,  $\boldsymbol{a}_{7\cdot 17}$ has $4=2\lfloor 17/7\rfloor$  gaps of each type.
 We observe that the boldface $1$, corresponding to  $\boldsymbol{\omega}_2^{3/7}[0]$, is the \emph{middle letter} or \emph{middle coefficient}.
 
 This example illustrates a particular case of Theorem \ref{thm:second gap}.
 }
\end{example}

\section{The second gap of \texorpdfstring{$\Phi_{pq}$ for $q \not\equiv \pm 1
\!\!\!\mod \!p$}{} } \label{section:the second gap}

 Our main result is Theorem~\ref{thm:second gap}, which determines the lengths and locations of  the first and second gaps of $\boldsymbol{a}_{pq}$, with no restrictions on $q \!\!\!\mod \! p$.
\begin{theorem}\label{thm:second gap} 
  Let $p<q$ be odd prime numbers, let $r = q \!\!\!\mod \! p$, and let  $m = \varphi(pq)$. 
 \begin{enumerate}
  \item If $p\geq 3$, then $ g_{1}(\boldsymbol{a}_{pq}) = p-1$. There are  $\lfloor q/p\rfloor $ such gaps  located  in     $\boldsymbol{a}_{pq}[0:q-1]$ and $\lfloor q/p\rfloor$ such gaps located in 
  $\boldsymbol{a}_{pq}[m-q+1:m]$.

  \medskip
  
  \item If $p\geq 5$ and $r \in \llbracket 2,p-2\rrbracket$, then   $g_{2}(\boldsymbol{a}_{pq}) = \max\{r-1,p-r-1\}$. There are at least $\lfloor q/p\rfloor $ such gaps located in $\boldsymbol{a}_{pq}[q: 2q-1]$ and at least $\lfloor q/p\rfloor $ such gaps located in  
  $\boldsymbol{a}_{pq}[m-2q-1: m-q]$.
 \end{enumerate}
\end{theorem}

In this section we deal with prime numbers   $5\le p < q$ and  $r  \in \llbracket 2,  p-2 \rrbracket$.  The cases $r = 1$ and $p-1$ (which in particular includes $p =3$) are treated separately in  Section~\ref{section:gaps and a representation theorem}.
 
 Our starting point is the representation of $\boldsymbol{a}_{pq}$ given in  Theorem~\ref{thm:representation theorem}. We study its gaps in three steps. Section~\ref{sec:gapsbasicwords} describes the gaps of the basic words $\boldsymbol{\omega}_{i}$. As illustrated in Example~ \ref{example:first example},  the gapset of $\boldsymbol{a}_{pq}$ coincides with that of its first half (Lemma~\ref{lemma:gap de a igual a gap c}). We then analyze new gaps arising from fractional powers and concatenation. An upper bound for gaps under concatenation is given in Section~\ref{sec:concatenationandgaps}. Finally, in Section~\ref{subsec:completing proof of theorem 1}, we combine this bound with the gaps from $\boldsymbol{\omega}_0^{q/p}$ and $\boldsymbol{\omega}_1^{q/p}$ to complete the proof.

\subsection{Gaps of basic words} \label{sec:gapsbasicwords}
 First, in Lemma~\ref{lemma:gap de a igual a gap c}, we show that since $\boldsymbol{a}_{pq}$ is reciprocal (it is equal to the word obtained by reversing its letters), it suffices to consider the basic words $\boldsymbol{\omega}_i$ for $i \in \llbracket 0, (p-3)/2 \rrbracket$ instead of $i \in \llbracket 0, p-2 \rrbracket$.

\begin{lemma}\label{lemma:gap de a igual a gap c}
Let $m=(p-1)(q-1)$, and let $\boldsymbol{c}_{pq} \in \mathcal{A}^{m/2}$ be the word
\[
\boldsymbol{c}_{pq}
=
\boldsymbol{\omega}_{0}^{q/p}
\boldsymbol{\omega}_{1}^{q/p}
\cdots
\boldsymbol{\omega}_{\frac{p-5}{2}}^{q/p}\,
\boldsymbol{\omega}_{\frac{p-3}{2}}^{\left(q-\frac{p-1}{2}\right)/p}.
\]
Then the gapset $G(\boldsymbol{a}_{pq})$ coincides with the gapset $G(\boldsymbol{c}_{pq})$.
\end{lemma}
 \begin{proof}
Since $\boldsymbol{a}_{pq}$ has odd length $m+1$ and each $\boldsymbol{\omega}_{i}^{q/p}$  has length $q$,  the length of  $\boldsymbol{c}_{pq}$ is
\[
 |\boldsymbol{c}_{pq}|   = \frac{p-3}{2}\, q + q - \frac{(p-1)}{2}=\frac{m}{2}. 
 \]
 If $z_{pq}$ denotes the middle letter of $\boldsymbol{a}_{pq}$, then $\boldsymbol{a}_{pq} = \boldsymbol{c}_{pq} z_{pq} {\boldsymbol{m}_{pq}}$, with ${\boldsymbol{m}_{pq}}$ of length $m/2$.
Since $\boldsymbol{a}_{pq}$ is reciprocal (see~\cite{Niven56} or~\cite{Cafure21}),   we have that
$\boldsymbol{m}_{pq}$ equals $ \widetilde{\boldsymbol{c}_{pq}}$. Hence $\boldsymbol{a}_{pq}=\boldsymbol{c}_{pq}z_{pq} \widetilde{\boldsymbol{c}_{pq}}$. 
Since $z_{pq} = \boldsymbol{a}_{pq}[m/2]$ is $1$ or $-1$ (see~\cite{LaLe99}), we conclude that $G(\boldsymbol{a}_{pq}) = G(\boldsymbol{c}_{pq})$.
\end{proof}

\subsubsection{Special values of the basic words}\label{section:properties of omegai}  The purpose of this section is to prove the  explicit formula for   $\boldsymbol{\omega}_i[j]$ given in Lemma~\ref{lemma:wieI}  and to compute the values of the basic words $\boldsymbol{\omega}_i$ at the positions $j=0,1,r-1,r,p-r,$ and $p-r+1$, which are  summarized in Table~\ref{table:values of omegai}.

Let $u_{+}, u_{-}\in\llbracket 1, p-1\rrbracket$ be the unique integers satisfying 
\[
  u_{+}q\equiv 1 \!\!\!\mod \!p\quad \text{ and } \quad u_{-}q\equiv p-1 \!\!\!\mod \!p.
 \]
 It follows that $u_{+} \in \llbracket 1,(p-1)/2\rrbracket$ if  and only if  $u_{-}\in \llbracket(p+1)/2,p-1\rrbracket$   since $u_{+}+u_{-}=p$. 

The next lemma  gives a formula for $\boldsymbol{\omega}_i[j]$ in terms of two maps ${\sf I}_{+}$ and ${\sf I}_{-}$ from $\llbracket 0, p-1 \rrbracket$ to $\llbracket 0, p-1 \rrbracket$:
\begin{equation}\label{eq:formula for I}
  {\sf I}_{+}(j) = -u_{+}j \!\!\! \mod \! p\quad \text{ and } \quad {\sf I}_{-}(j) = u_{-}(j -1)\!\!\! \mod \! p . 
 \end{equation}
The choice of notation ${\sf I}_{+}$ and ${\sf I}_{-}$ for these maps is motivated by the identities \eqref{eq:diI}.

\begin{lemma}\label{lemma:wieI} For  $i,j \in \llbracket 0, p-1 \rrbracket$, the following holds:
   \[
\boldsymbol{\omega}_i[j]=
\begin{cases} \phantom{-}1, & \text{if } \,  {\sf I}_{+}(j)\le i <{\sf I}_{-}(j)\\[1ex]
-1,&  \text{if } \, {\sf I}_{-}(j)\le i <{\sf I}_{+}(j)\\[1ex]
\phantom{-}0, & \text{otherwise}. 
\end{cases}
\]
\end{lemma}

\begin{proof}
Let $\{\boldsymbol{d}_{i}\}_{0\le i\le p-1}$ be the sequence of words in $\mathcal{A}^{p}$ defined in  \eqref{defi:definition of di}. 
Since $\boldsymbol{d}_{i} = \sigma^{iq}(\boldsymbol{d}_{0})$, we have that  $\boldsymbol{d}_{i}[j]=\boldsymbol{d}_{0}[j+iq \!\!\! \mod \! p]$. For  $i,j \in \llbracket 0, p-1 \rrbracket$, the following hold: 
\[\boldsymbol{d}_{i}[j]=1\  \iff \  j+iq\equiv 0 \!\!\!\mod \!p \quad \text{and}\quad
\boldsymbol{d}_{i}[j]=-1 \ \iff \  j+iq\equiv 1  \!\!\!\mod \!p.
\]
 Hence, for $i\in \llbracket 0,p-1\rrbracket$,
 \begin{align}
\label{eq:diI} 
\begin{split}
\boldsymbol{d}_i[j]&=1 \quad \phantom{-}\text{if and only if }\quad\ i={\sf I}_{+}(j), 
\\[1ex] 
  \boldsymbol{d}_i[j]&=-1\quad \text{if and only if }\quad i={\sf I}_{-}(j).
\end{split}
\end{align}
For any other $i,j$, we have $\boldsymbol{d}_i[j]=0$.

Now we turn to the basic words. From \eqref{eq:definition of omega_i}, we have that for each $(i,j) \in \llbracket 0, p-1\rrbracket^{2}$, 
\begin{equation}\label{eq:omegai suma}
\boldsymbol{\omega}_{i}[j]=\sum_{k=0}^i \boldsymbol{d}_{k}[j].
\end{equation}
Since ${\sf I}_{+}$ and ${\sf I}_{-}$ are bijective, at most two values of $\boldsymbol{d}_k[j]$ can be different from zero: $k={\sf I}_{+}(j)$ or $k={\sf I}_{-}(j)$. Considering the relative ordering of $i$,  ${\sf I}_{+}(j)$, and ${\sf I}_{-}(j)$, the result follows. 
\end{proof}

In Table~\ref{table:values of I}, we collect some relevant values of ${\sf I}_{+}$ and ${\sf I}_{-}$. Combining these results with the information provided by Lemma~\ref{lemma:wieI} about the letters $\omega_{i}[j]$, Example~\ref{example:calculo de omegau+r-1} below shows how to compute some of the values displayed in Table~\ref{table:values of omegai}.

\begin{table}[H]
\setlength{\tabcolsep}{4pt}
\centering
 \caption{Values of ${\sf I}_{+}$ and ${\sf I}_{-}$}
 \label{table:values of I}

 \begin{tabular}{|c|c|c|c|c|c|c|c|c|c|}
 \hline
               & $0$   & $1$      &$r-1$            & $r$           &$p-r$             & $p-r+1$   &$p-1$
\\ \hline  
  ${\sf I}_{+}$ & $0$  & ${u_{-}}$ &$u_{+}-1 $  & $p-1$         &$1$               &  $1+u_{-}\!\!\!\mod \!p$ & $u_{+}$
  \\
  ${\sf I}_{-}$ & $u_{+}$ & $0$      & $2u_{+}-1 \!\!\!\mod \!p$&$u_{+}-1$& $ 1+u_{+}\!\!\!\mod \!p$&   $1$      & $2u_{+}\!\!\!\mod \!p$   
  \\ \hline    
 \end{tabular}
 \end{table}

\begin{example}\label{example:calculo de omegau+r-1}{\rm We show how to obtain the information stated in the first, second, and fourth rows of the Table~\ref{table:values of omegai}.

\noindent \emph{First and second rows.} Table~\ref{table:values of I} shows that  ${\sf I}_{+}(0)= 0$ and ${\sf I}_{-}(0)=u_{+}$. Therefore we deduce that 
\[
 \boldsymbol{\omega}_{i}[0] = 
 \begin{cases}
  1, & \text{if } \, 0\leq i < u_{+}
  \\[1ex]
  0, & \text{if } \, u_{+}\leq i \leq p-1.
 \end{cases}
\]
\noindent \emph{Fourth row.} According to Table~\ref{table:values of I}, 
\[{\sf I}_{+}(r-1)=u_{+}-1 \quad \text{and}\quad {\sf I}_{-}(r-1)=2u_{+}-1 \!\!\!\mod \!p.
\]
Since $u_{+}\le (p-1)/2$, it follows that ${\sf I}_{-}(r-1)=2u_{+}-1 \le p-2$. Thus, ${\sf I}_{+}(r-1)\leq u_{+} -1 < {\sf I}_{-}(r-1)$ and the result follows from the first line of the formula given in Lemma~\ref{lemma:wieI}. }
\end{example} 

\vspace{-0.5cm}
\begin{table}[ht]
\centering
\caption{The first column lists  $\boldsymbol{\omega}_i[j]$ for special values of $j$. The second column gives the corresponding conditions on $i$, and the third column indicates extra conditions on $p$, $u_{+}$, and $r$ if applicable.}
 
 \begin{tabular}{|rcl|c|c|} 
 \hline
 $ \boldsymbol{\omega}_i[0]$&$\!\!\!=$&$\!\!\!1$      & for $0\le i <u_{+}$       &    \\  \hline 
 $\boldsymbol{\omega}_i[0]$&$\!\!\!=$&$\!\!\!0$      & {for} $u_{+}\le i \le p-1$  &    \\  \hline  $\boldsymbol{\omega}_i[1]$&$\!\!\!=$&$\!\!\!-1$ &     {for} $0\le i <u_{-}$     &    \\  \hline
 $\boldsymbol{\omega}_{u_{+}-1}[r-1]$&$\!\!\!=$&$\!\!\! 1$ &                                & $u_{+}\le (p-1)/2$   \\  \hline
$\boldsymbol{\omega}_i[r]$&$\!\!\!=$&$\!\!\!-1$      &  for  $u_{+}-1\le i < p-1$       &                     \\  \hline
 $\boldsymbol{\omega}_i[p-r]$&$\!\!\!=$&$\!\!\!1$      &   {for} $1\le i < 1 + u_{+}$         &   $r\neq p-1$, thus $u_{+}\neq p-1$.                 \\  \hline
 $\boldsymbol{\omega}_i[p-r+1]$&$\!\!\!=$&$\!\!\!-1$      &   {for} $1\le i <  1 + u_{-}$  &       $r\ge 2$               \\  \hline
 \end{tabular}
\medskip
 \label{table:values of omegai}
 \end{table}
\vspace{-0.4cm}

\subsubsection{Alternation and gaps}
From Table~\ref{table:values of omegai} we obtain the bound 
$g_{1}(\boldsymbol{\omega}_i)\le \max\{r,p-r\}$ for $1\le i<\max\{u_+,u_{-}\}$.
The goal of this section is to prove Lemma~\ref{lemma:gaps control}, which yields the sharper bound in Proposition~\ref{prop:gap of words wi}.
To this end, in Lemma~\ref{lemma:alternate signs} we establish a criterion for words over $\mathcal{A}=\{-1,0,1\}$ to have alternating $1$s and $-1$s, and to compute their last nonzero letter.  We then show, mainly via Lemma~\ref{lemma:wi satisface el criterio}, that each $\boldsymbol{\omega}_{i}$ satisfies this criterion.  
We do not rely on  the known alternation of the nonzero coefficients $1$s and $-1$s of binary cyclotomic polynomials (see e.g. ~\cite{Migotti83},~\cite{Lenstra79},~\cite{LaLe99}). Sign alternation has previously been used to prove  that $g_{1}(\Phi_{pq}) = p-1$ in~\cite{Camburu16}.

\begin{lemma}\label{lemma:alternate signs}
 For
  $\boldsymbol{u}\in \mathcal{A}^{+}$, let $S_{\boldsymbol{u}}:\llbracket 0,|\boldsymbol{u}|-1\rrbracket\to \mathbb{Z}$  be  the mapping  defined by $ S_{\boldsymbol{u}}(j)=\sum_{t=0}^j \boldsymbol{u}[t]$.     
 Let $V(S_{\boldsymbol{u}}) = \{S_{\boldsymbol{u}}(j): j\in  \llbracket 0,|\boldsymbol{u}|-1\rrbracket \}$ denote its value set. 
\begin{enumerate}

\smallskip

 \item \label{lemma:alternate signs ii}If $S_{\boldsymbol{u}}(|\boldsymbol{u}|-1)=0$ and $V(S_{\boldsymbol{u}})=\{0,1\}$, then $1$ and $-1$ alternate in $\boldsymbol{u}$ and  the last nonzero letter of $\boldsymbol{u}$ is $-1$.
 
 \item \label{lemma:alternate signs iii}If $S_{\boldsymbol{u}}(|\boldsymbol{u}|-1)=0$ and $V(S_{\boldsymbol{u}})=\{-1,0\}$, then $1$ and $-1$ alternate in $\boldsymbol{u}$ and  the last nonzero letter of $\boldsymbol{u}$ is $1$.
\end{enumerate}
\end{lemma}
 
\begin{proof}
    We only prove $ (1)$, since the proof of $ (2)$ proceeds similarly. From $V(S_{\boldsymbol{u}})=\{0, 1\}$, we deduce that the first nonzero letter of $\boldsymbol{u}$ is $1$ and that there are no two consecutive $1$s. The fact that $S_{\boldsymbol{u}}(|\boldsymbol{u}|-1)=0$ implies that the second nonzero letter is $-1$. An inductive argument shows the validity of the assertion.
\end{proof}

The next lemma shows that the words $\boldsymbol{\omega}_i$ and $\boldsymbol{\omega}_i^{r/p}$ satisfy the assumptions on the value sets required by Lemma~\ref{lemma:alternate signs}.
For $0\le i\le p-2$, we consider the   mappings     
\begin{align*}
S_{\boldsymbol{\omega}_i}(j) & = \sum_{t=0}^{j} \boldsymbol{\omega}_{i}[t],\quad \text{for}\ j\in\llbracket 0,p-1\rrbracket,
\\
S_{\boldsymbol{\omega}_i^{r/p}}(j) & =\sum_{t=0}^{j} \boldsymbol{\omega}_{i}^{r/p}[t]=S_{\boldsymbol{\omega}_i}(j),\quad \text{for}\ j\in\llbracket 0,r-1\rrbracket.
\end{align*}

\begin{lemma}\label{lemma:wi satisface el criterio}  The following holds:
\begin{enumerate}
 \item \label{lemma:wi satisface el criterio 3}  $V(S_{\boldsymbol{\omega}_i})=\{0,1\}$ for $0\le i<u_{+}$.
\item \label{lemma:wi satisface el criterio 4}   $V(S_{\boldsymbol{\omega}_i})=\{-1,0\}$ for $u_{+}\le i \le p-2$. 
\item \label{lemma:wi satisface el criterio 5} If $u_{+}\ge 2$, then   $V\left(S_{\boldsymbol{\omega}_i^{r/p}}\right)=\{0,1\}$ for $0\le i\le u_{+}-2$.
\item \label{lemma:wi satisface el criterio 6} If  $u_{+}<u_{-}$, then $V\left(S_{\boldsymbol{\omega}_i^{r/p}}\right)=\{-1,0\}$ for $u_{+}\le i < u_{-}$.
\end{enumerate}

\end{lemma}
\begin{proof}
We first prove the next expression for $S_{\boldsymbol{\omega}_{i}}(j)$  using indicator functions:
\begin{equation}\label{lemma:map S como diferencia de indicadoras}
 S_{\boldsymbol{\omega}_{i}}(j) =\boldsymbol{1}_A(i,j)-\boldsymbol{1}_B(i,j)\quad \text{for}\ 0\le i,j\le p-1,
\end{equation}
where \
\[A  =\left\{(i,j)\in \llbracket 0, p-1\rrbracket^2 : {\sf I}_{+}(j)\le i\right\}\quad \text{ and } \quad
 B =\left\{(i,j)\in  \llbracket 0, p-1\rrbracket^2 : {u}_{+} \le i\right\}.\] 
Recalling \eqref{eq:omegai suma} we obtain 
\begin{equation} \label{eq:Sij en funcion de Somega}
S_{\boldsymbol{\omega}_{i}}(j)= \sum_{t=0}^{j} \boldsymbol{\omega}_{i}[t]=\sum_{t=0}^{j}  \sum_{k=0}^{i}\boldsymbol{d}_{k}[t]=\sum_{k=0}^{i}  \sum_{t=0}^{j}\boldsymbol{d}_{k}[t]=\sum_{k=0}^{i}S_{\boldsymbol{d}_{k}}(j).
\end{equation}

First we consider the case $k= u_{+}$. Table~\ref{table:values of I} shows that ${\sf I}_{-}(0)= {\sf I}_{+}(p-1)=u_{+}$. Thus  $\boldsymbol{d}_{u_{+}}=(-1)0^{p-2}1$ and hence  $S_{\boldsymbol{d}_{u_{+}}}(j)$ is the negative of an indicator function:
\[
S_{\boldsymbol{d}_{u_{+}}}(j)=  
\begin{cases}
-1, & \text{if}\, j\in \llbracket 0,p-2\rrbracket
\\[1ex]
\phantom{-} 0, & \text{if}\, j = p-1.
\end{cases} = -\boldsymbol{1}_{\{0\le j\le p-2\} }(j).
\]

Next we consider $k \neq  u_{+}$. Since ${\sf I}_{+}$ is bijective and ${\sf I}_{+}(p-1)=u_{+}$, there exists an unique $j\in \llbracket 0,p-2\rrbracket$ such that $k={\sf I}_{+}(j)$. In this way,  $\boldsymbol{d}_{k}[j]=1$, $\boldsymbol{d}_{k}[j+1]=-1$ and $\boldsymbol{d}_{k}[t]=0$ for $t\neq j $ or $j+1$. These facts imply that  $S_{\boldsymbol{d}_{k}}(j)$ is an indicator function for $k\neq u_{+}$:
\[
S_{\boldsymbol{d}_{k}}(j)=  \begin{cases}
1, &  \text{if }\, k = {\sf I}_{+}(j)
\\[1ex]
0, & \text{if }\, \text{otherwise}.
\end{cases} \, = \, \boldsymbol{1}_{\{(k,j):k={\sf I}_+(j)\}}(k,j).
\]
Next,  we compute $S_{\boldsymbol{\omega}_i}(j)$ as follows:
 \begin{align*}S_{\boldsymbol{\omega}_{i}}(j)=\sum_{k=0}^i S_{\boldsymbol{d}_{k}}(j)&=\sum_{\substack{0\le k\le i\\[0.5ex]k\neq u_+}}S_{\boldsymbol{d}_{k}}(j)+{  S}_{\boldsymbol{d}_{u_{+}}}(j)\boldsymbol{1}_{\{u_+\le i\}}(i)\\
 &=\sum_{\substack{0\le k\le i\\[0.5ex]k\neq u_+}}\boldsymbol{1}_{\{(k,j):k={\sf I}_+(j)\}}(k,j) -\boldsymbol{1}_{\{0\le j\le p-2\} }(j)\boldsymbol{1}_{\{u_+\le i\}}(i) \\
 &=\boldsymbol{1}_{\{(i,j):{\sf I}_+(j)\le i,\ {\sf I}_+(j)\neq u_+\}}(i,j) -\boldsymbol{1}_{\{0\le j\le p-2\}  }(j)\boldsymbol{1}_{\{u_+\le i\}}(i) .
 \intertext{Since ${\sf I}_+(p-1)=u_+$, we obtain}
 S_{\boldsymbol{\omega}_{i}}(j)&=\boldsymbol{1}_{\{(i,j):{\sf I}_+(j)\le i\}}(i,j) -\boldsymbol{1}_{\{0\le j\le p-1\}  }(j)\boldsymbol{1}_{\{u_+\le i\}}(i)
 \\[1ex] &=\boldsymbol{1}_A(i,j)-\boldsymbol{1}_B(i,j).
 \end{align*} 
This completes the proof of \eqref{lemma:map S como diferencia de indicadoras}.

\smallskip
\noindent \emph{Proof of (\ref{lemma:wi satisface el criterio 3})} and (\ref{lemma:wi satisface el criterio 4}). For $0\le i<u_+$ and $0\le j\le p-1$,  \eqref{lemma:map S como diferencia de indicadoras} implies that $S_{\boldsymbol{\omega}_i}(j)=\boldsymbol{1}_A(i,j)$ and thus $V(S_{\boldsymbol{\omega}_i})=\{0,1\}$.

For $u_{+}\le i \le p-2$ and $0\le j\le p-1$, \eqref{lemma:map S como diferencia de indicadoras} implies that $
S_{\boldsymbol{\omega}_i}(j)=\boldsymbol{1}_A(i,j)-1$ and thus $V(S_{\boldsymbol{\omega}_i})=\{-1,0\}$.

\smallskip
\noindent \emph{Proof of  (\ref{lemma:wi satisface el criterio 5}) and (\ref{lemma:wi satisface el criterio 6}).} Our conditions on $i$ imply that $i <\max\{u_{+},u_{-}\}$. Since $r\ge 2$, rows 1 and 3 of Table~\ref{table:values of omegai}  imply that each $\boldsymbol{\omega}_i^{r/p}$ is nonzero.  Table~\ref{table:values of I} shows that ${\sf I}_{+}(r-1)=u_{+}-1$. Thus,  if $i\le u_{+}-2$, then    $(i,r-1)\notin A\cup B$;  and if $i\ge u_{+}$, then   $(i,r-1)\in A\cap B$.  
     In any case,  \eqref{lemma:map S como diferencia de indicadoras} implies that 
\begin{equation}\label{eq:Swirp=-}S_{\boldsymbol{\omega}^{r/p}_i}(|\boldsymbol{\omega}^{r/p}_{i}|-1)=S_{\boldsymbol{\omega}_i}(r-1)=\boldsymbol{1}_{A}(i,r-1) - \boldsymbol{1}_{B}(i,r-1) = 0.
\end{equation}
 Thus, $0\in V(S_{\boldsymbol{\omega}_i^{r/p}}) \subseteq V(S_{\boldsymbol{\omega}_i })$
and since $\boldsymbol{\omega}_i^{r/p}$ is nonzero (see    rows 1 and 2 of Table~\ref{table:values of omegai}), we conclude that these two sets are equal. Thus, from what we proved in (\ref{lemma:wi satisface el criterio 3}) and 
(\ref{lemma:wi satisface el criterio 4}), we conclude that if  $0\le i \le u_{+}-2$, then $V(S_{\boldsymbol{\omega}_i^{r/p}})=\{0,1\}$;   and if $    u_{+}\le i < u_{-}$, then  $V(S_{\boldsymbol{\omega}_i^{r/p}})=\{-1,0\}$.
\end{proof}

Assertions (\ref{lemma:wi satisface el criterio 5}) and (\ref{lemma:wi satisface el criterio 6}) do not cover the case $i=u_+-1$.
 For example, if $p=7$ and  $q=17$, then $r=3$, $u_+=5$, and  $\boldsymbol{\omega}_4^{3/7}=100$. Thus $V(S_{\boldsymbol{\omega}_4^{3/7}}) = \{1\}$ (see Example~\ref{example:first example}). 

The next lemma gives information about the location of the nonzero letters within the basic words. It follows from Table~\ref{table:values of omegai}, Lemma~\ref{lemma:alternate signs}, and Lemma~\ref{lemma:wi satisface el criterio}.
  
\begin{lemma}\label{lemma:gaps control} Let $5 \leq p < q$ be prime numbers  such that $r\in \llbracket 2, p-2 \rrbracket$. 
\begin{enumerate}
\item \label{gaps i} If $1\le i <u_{+}$, then the following assertions hold. 

\begin{enumerate}
 \item \label{gaps ib} There exist  $\ell_{1} \in \llbracket 1, p-r-1 \rrbracket$ and $\ell_{2} \in \llbracket p-r+1, p-1 \rrbracket $ such that
$\boldsymbol{\omega}_i[\ell_1]= \boldsymbol{\omega}_i[\ell_2]=-1$.

\item \label{gaps ic} If $1\le i\le u_{+}-2$, there exists  $\ell_{3} \in \llbracket 1, r-1 \rrbracket$ such that
$\boldsymbol{\omega}_i[\ell_3]=-1$.
\end{enumerate}

\item \label{gaps ii} If $u_{+} \le i <u_{-}$, then $r\ge 3$, and the following assertions hold.

\begin{enumerate}    
 \item \label{gaps iib} There exist $\ell_{4} \in \llbracket 2, p-r \rrbracket$  and $\ell_{5} \in \llbracket p-r+2, p-1 \rrbracket$ such that  
$\boldsymbol{\omega}_i[\ell_4]= \boldsymbol{\omega}_i[\ell_5]=1$.

\item \label{gaps iic} There exists $\ell_{6}  \in \llbracket 2, r-1 \rrbracket$ such that 
$\boldsymbol{\omega}_i[\ell_6]=1$.
\end{enumerate}
\end{enumerate}
\end{lemma}
\begin{proof}
The proof follows by verifying that  each $\boldsymbol{\omega}_{i}$ satisfy the assumptions of  Lemma~\ref{lemma:alternate signs}.  

Since $S_{\boldsymbol{d}_{k}}(p-1)=0$, identity \eqref{eq:Sij en funcion de Somega} yields $S_{\boldsymbol{\omega}_{i}}(|\boldsymbol{\omega}_{i}|-1 )=0$. By Lemma~\ref{lemma:wi satisface el criterio},  $\boldsymbol{\omega}_1, \dots, \boldsymbol{\omega}_{p-2} $ are nonzero and satisfy the value-set condition. Thus, the assumptions of Lemma~\ref{lemma:alternate signs} hold.

For $0\le i<u_{+}$,     
  the letters $1$ and $-1$ alternate and   the first nonzero letter of $\boldsymbol{\omega}_{i}$ is $1$ and its last nonzero letter is  $-1$. Recalling that $\boldsymbol{\omega}_{i}[0] = \boldsymbol{\omega}_i[p-r] = 1$, we prove item (\ref{gaps ib}). 

Similarly, if $u_{+} \le i <u_{-}$, 
  the  letters $-1$ and $1$ alternate and   the first nonzero letter of $\boldsymbol{\omega}_{i}$ is $-1$ and its last nonzero letter is  $1$. Since $\boldsymbol{\omega}_{i}[1] = \boldsymbol{\omega}_i[p-r+1] = -1$, we prove item (\ref{gaps iib}). 

 Assertions (\ref{gaps ic}) and (\ref{gaps iic}) 
follows with the same reasoning for $i\neq u_{+}-1$. Recall that we have already proved that $S_{\boldsymbol{\omega}_i^{r/p}}(r-1)=0$  in \eqref{eq:Swirp=-}. 

The case $i=u_{+}-1$ follows from Table~\ref{table:values of omegai}, which shows that $\boldsymbol{\omega}_{u_{+}-1}[r-1] = 1$ when $u_{+} \leq (p-1)/2$. This last condition holds because $u_{+}<u_{-}$. 

\end{proof}

We now illustrate the situation described in Table~\ref{table:values of omegai} and   Lemma~\ref{lemma:gaps control}(\ref{gaps ib},  \ref{gaps iib}):

\begin{itemize}  
\item \noindent If $1\le i< u_{+}$, then $\boldsymbol{\omega}_i=\underbrace{1}_{0}\cdots \underbrace{(-1)}_{\ell_1}\cdots \underbrace{1}_{p-r}\cdots \underbrace{(-1)}_{\ell_2} \cdots$
 
 \smallskip
\item \noindent If $u_{+} \leq i < u_{-}$, then $\boldsymbol{\omega}_i=\underbrace{0}_{0}\underbrace{(-1)}_{1}\cdots \underbrace{1}_{\ell_{4}}\cdots \underbrace{(-1)}_{p-r+1}\cdots \underbrace{1}_{\ell_{5}} \cdots$
\end{itemize}

\begin{proposition}\label{prop:gap of words wi}
Let $5 \leq p < q$ be prime numbers and  $r\in \llbracket 2, p-2 \rrbracket$. The  following assertions hold. 
\begin{enumerate}
\item \label{prop:gap of words w01} $\boldsymbol{\omega}
_0=1(-1)0^{p-2}$ and $ \boldsymbol{\omega}_{1}=1(-1)0^{p-r-2}1(-1)0^{r-2}$.
\smallskip
\item  \label{prop:gap of words wu}
If $1\leq i \leq \max\{u_{+}, u_{-}\}$, then $g_{1}(\boldsymbol{\omega}_i)\le \max  \{p-r-1,r-1\}$.

\end{enumerate}

\end{proposition}

\begin{proof}   
Assertion  (\ref{prop:gap of words w01}) for $\boldsymbol{\omega}_0$ is immediate. Since $r$ and $p-r$  are at least 2,   
  \[\boldsymbol{\omega}_1= \boldsymbol{d}_{0}+\boldsymbol{d}_{1}=1(-1)0^{p-2}+0^{p-r}1(-1)0^{r-2}=1(-1)0^{p-r-2}1(-1)0^{r-2}.\]
Assertion (\ref{prop:gap of words wu}) is immediate from (\ref{gaps ib}) and (\ref{gaps iib}) in Lemma
\ref{lemma:gaps control}.  
\end{proof}

\subsection{Concatenation and gaps}\label{sec:concatenationandgaps} In this section we establish an upper bound on the gaps of the concatenation of $\boldsymbol{\omega}_i^{q/p}$ with $i\in\llbracket 1,(p-3)/2\rrbracket$. 
We first bound the gaps of $\boldsymbol{\omega}_i^{q/p}$, then those arising from the concatenation of two such words (see Lemma~\ref{lemma:concatenation of two basic words to the power}), and
finally, prove the desired upper bound in Proposition~\ref{prop:hasta p-3/2}.

The maximum gap $g_{1}(\boldsymbol{u}\boldsymbol{v})$ can exceed both $g_{1}(\boldsymbol{u})$ and $g_{1}(\boldsymbol{v})$, depending on the positions of their first and last nonzero letters. This motivates the following definition.

 The \emph{beginning} and \emph{end} of a nonzero word $\boldsymbol{u}\in \mathcal{A}^{+}$ are the positions of its first and last nonzero letters, respectively:
\[
\com{\boldsymbol{u}}=\min\{j: \boldsymbol{u}[j]\neq 0\}\quad \text{and}\quad \fin{\boldsymbol{u}}=\max\{j: \boldsymbol{u}[j]\neq 0\}.
\]
The following assertions hold:
\begin{align}
\label{eq:begendsuv}\com{\boldsymbol{uv}} 
 & =\com{\boldsymbol{u}} \quad \text{ and } \quad 
\fin{\boldsymbol{uv}}=|\boldsymbol{u}| + \fin{\boldsymbol{v}},
\\
\label{eq:gaps uv}
g_1(\boldsymbol{u}\boldsymbol{v}) &=\max\big\{g_1(\boldsymbol{u}),g_1(\boldsymbol{v}), |\boldsymbol{u}|+\com{\boldsymbol{v}}-\fin{\boldsymbol{u}} \big \} .
\end{align}

The following lemma provides upper bounds for $g_{1}(\boldsymbol{\omega}_i^{q/p})$ and $g_{1}(\boldsymbol{\omega}_i^{q/p}\boldsymbol{\omega}_{i+1}^{q/p})$.

\begin{lemma}\label{lemma:concatenation of two basic words to the power}
Let $5 \leq p < q$ be prime numbers and let $r \in \llbracket 2, p-2 \rrbracket$. 
 \begin{enumerate}
\item \label{lemma:concatenation of two basic words to the power 1} If $1\le i< \max\{u_{+},u_{-}\}$, then 
\[g_1(\boldsymbol{\omega}_i^{q/p} )\le\max\{p-r-1,r-1\}.
\]
\item \label{lemma:concatenation of two basic words to the power 2} If 
$1\le i<u_{+}-1$ or $u_{+}\le i<u_{-}-1$, then
 \[g_1(\boldsymbol{\omega}_i^{q/p}\boldsymbol{\omega}_{i+1}^{q/p})\le \max\{p-r-1,r-1\}.\]
  \end{enumerate}
\end{lemma}

\begin{proof} First recall that both $\boldsymbol{\omega}_{i}$ and $\boldsymbol{\omega}_{i}^{r/p}$ are nonzero words (see Lemma~\ref{lemma:wi satisface el criterio} and Table~\ref{table:values of omegai} for $u_{+}-1$).
Since $q > p$, from \eqref{eq:gaps uv}  
and Proposition~\ref{prop:gap of words wi} we obtain 
\[g_1(\boldsymbol{\omega}_{i}^{q/p}) \leq \max\{p-r-1, r-1,|\boldsymbol{\omega}_{i}|+\com{\boldsymbol{\omega}_{i}}-\fin{\boldsymbol{\omega}_{i}}\}.\]

We now prove the first assertion of the lemma. If $1\le i < u_{+}$, then $\com{\boldsymbol{\omega}_i}=0$ by Table \ref{table:values of omegai}. Moreover,  Lemma~\ref{lemma:gaps control} implies  $ \fin{\boldsymbol{\omega}_i}\ge p-r+1$. Thus,
\begin{align*}
|\boldsymbol{\omega}_{i}|+\com{\boldsymbol{\omega}_{i}} - \fin{\boldsymbol{\omega}_{i}}\le r-1. 
\end{align*}
If   $u_{+}\le i<u_{-}$, then  $\com{\boldsymbol{\omega}_i}=1$ by Table \ref{table:values of omegai}. Moreover, Lemma~\ref{lemma:gaps control} implies  $\fin{\boldsymbol{\omega}_i}\ge         p-r+2$. Thus,
\[ |\boldsymbol{\omega}_{i}|+\com{\boldsymbol{\omega}_{i}} - \fin{\boldsymbol{\omega}_{i}}\le r-1.\]     
The validity of \eqref{lemma:concatenation of two basic words to the power 1} follows.
 
We now prove the second assertion. From \eqref{eq:gaps uv} and the first part of this lemma 
\begin{align*}
g_1(\boldsymbol{\omega}_i^{q/p}\boldsymbol{\omega}_{i+1}^{q/p}) & = \max\{g_1(\boldsymbol{\omega}_{i}^{q/p}),g_1(\boldsymbol{\omega}_{i+1}^{q/p}), q +\com{\boldsymbol{\omega}_{i+1}^{q/p}}-\fin{\boldsymbol{\omega}_{i}^{q/p}} \}
\\[1ex]
& \leq \max\{p-r-1, r-1, q +\com{\boldsymbol{\omega}_{i+1}^{q/p}}-\fin{\boldsymbol{\omega}_{i}^{q/p}} \}.
\end{align*}
Since $q > p$,  from \eqref{eq:begendsuv} we deduce that 
\[
\com{\boldsymbol{\omega}_{i+1}^{q/p}}=\com{\boldsymbol{\omega}_{i+1}} \quad \mbox{ and } \quad 
\fin{\boldsymbol{\omega}_{i}^{q/p}}=q-r+\fin{\boldsymbol{\omega}_{i}^{r/p}}.
\]

\noindent If   $1\le i<u_{+} -1$, then  $\com{\boldsymbol{\omega}_{i+1}^{q/p}}=0$ and ${\fin{\boldsymbol{\omega}_{i}^{r/p}} \ge 1}$ by Lemma~\ref{lemma:gaps control}. Thus,
\[  
q +\com{\boldsymbol{\omega}_{i+1}^{q/p}}-\fin{\boldsymbol{\omega}_{i}^{q/p}} \leq 
 q- (q-r+1) = r- 1.\]
\noindent If  $u_{+}\le i<u_{-} -1$, then $\com{\boldsymbol{\omega}_{i+1}^{q/p}}=1$ and ${\fin{\boldsymbol{\omega}_{i}^{r/p}}\ge 2}$ by Lemma~\ref{lemma:gaps control}. Thus,
\[  
q +\com{\boldsymbol{\omega}_{i+1}^{q/p}}-\fin{\boldsymbol{\omega}_{i}^{q/p}} \leq 
 q + 1- (q-r+2) = r- 1.\]
In both cases, the bound holds, completing the proof.
\end{proof}

Since we know that $g_{1}(\boldsymbol{\omega}_0) = p-1$, the upper bound provided by the next proposition is an upper bound for $g_{2}(\boldsymbol{c}_{pq})$. 

\begin{proposition} \label{prop:hasta p-3/2}
Let $5 \leq p < q$ be prime numbers and let $r \in \llbracket 2, p-2 \rrbracket$. Then
\[g_1\left(\boldsymbol{\omega}_1^{q/p}\boldsymbol{\omega}_2^{q/p} \cdots \, \boldsymbol{\omega}_{\frac{p-3}{2}}^{q/p}\right)\le \max\{p-r-1,r-1\}.\] 
\end{proposition}
\begin{proof}
First, observe that for any nonzero words $\boldsymbol{u}_{1}$, $\boldsymbol{u}_2$,\ldots, $\boldsymbol{u}_{n}$ in $\mathcal{A}^{+}$, it follows that
\begin{equation}\label{eq:conca}
g_1(\boldsymbol{u}_{1}\boldsymbol{u}_{2}\cdots \boldsymbol{u}_{n})=\max\{g_1(\boldsymbol{u}_{1}\boldsymbol{u}_{2}), g_1(\boldsymbol{u}_{2}\boldsymbol{u}_{3}), \ldots, g_{1}(\boldsymbol{u}_{n-1}\boldsymbol{u}_{n})\}.
\end{equation}

\noindent \emph{Case $u_{+}\ge (p-1)/2$}. Since $u_{+}-1\ge (p-3)/2$, the second assertion of Lemma~\ref{lemma:concatenation of two basic words to the power} holds for $\boldsymbol{\omega}_i$, whenever $i \in \llbracket 1, (p-3)/2\rrbracket$. The result then follows from \eqref{eq:conca}. 

\smallskip

 \noindent \emph{Case $u_{+}\le (p-3)/2$}. Since $u_{-}\ge (p+3)/2$, Lemma~\ref{lemma:concatenation of two basic words to the power} together with \eqref{eq:conca}
 yield 
\begin{align}
\label{eq:gaps dos partes}
\begin{split}
g_1\left(\boldsymbol{\omega}_1^{q/p}\cdots \boldsymbol{\omega}_{u_{+}-1}^{q/p}\right) &\le \max\{p-r-1,r-1\}, \\[1ex] 
g_1\left(\boldsymbol{\omega}_{u_{+}}^{q/p}\cdots \boldsymbol{\omega}_{u_{-}-1}^{q/p}\right)&\le \max\{p-r-1,r-1\}.
\end{split}
\end{align}

Since $u_{-}-1\ge (p+1)/2$, the result follows by appealing to~\eqref{eq:gaps uv} to compute the gap of $
\boldsymbol{\omega}_1^{q/p}\cdots \boldsymbol{\omega}_{u_{+}-1}^{q/p} \boldsymbol{\omega}_{u_{+}}^{q/p}\cdots \boldsymbol{\omega}_{u_{-}-1}^{q/p}$. 
According to Table \ref{table:values of omegai} and \eqref{eq:begendsuv}, 
\[\com{\boldsymbol{\omega}_{u_{+}}^{q/p}\cdots \boldsymbol{\omega}_{u_{-}-1}^{q/p}}=\com{\boldsymbol{\omega}_{u_{+}}}= 1.\]
Moreover, Table~\ref{table:values of omegai} shows that $\boldsymbol{\omega}_{u_{+}-1}[r-1]\neq 0$, so $\fin{\boldsymbol{\omega}_{u_{+}-1}^{r/p}} = r-1$. Therefore, from~\eqref{eq:begendsuv},
\begin{align*}
\fin{\boldsymbol{\omega}_1^{q/p}\cdots \boldsymbol{\omega}_{u_{+}-1}^{q/p}} & =(u_{+}-2)q+\fin{\boldsymbol{\omega}_{u_{+}-1}^{q/p}}
\\
&=(u_{+}-1)q-r+\fin{\boldsymbol{\omega}_{u_{+}-1}^{r/p}}
\\ 
&=(u_{+}-1)q-1.
\end{align*}
Finally, the assertion of the proposition holds since
\[|\boldsymbol{\omega}_1^{q/p}\cdots \boldsymbol{\omega}_{u_{+}-1}^{q/p}|+\com{\boldsymbol{\omega}_{u_{+}}^{q/p}\cdots \boldsymbol{\omega}_{u_{-}-1}^{q/p}}-\fin{\boldsymbol{\omega}_1^{q/p}\cdots\boldsymbol{\omega}_{u_{+}-1}^{q/p}}=2.\]
\end{proof}

\subsection{End of the proofs of Theorem~\ref{thm:second gap} and Theorem~\ref{thm:teorema 1}} \label{subsec:completing proof of theorem 1} 

In Lemma~\ref{lemma:gap de a igual a gap c} we proved that the gapsets $G(\boldsymbol{a}_{pq})$ and $G(\boldsymbol{c}_{pq})$ are equal. We write 
\[\boldsymbol{c}_{pq}=\boldsymbol{\omega}_0^{q/p}\boldsymbol{\omega}_1^{q/p}\boldsymbol{b}_{pq}\quad
\mbox{ where }\quad \boldsymbol{b}_{pq}=\boldsymbol{\omega}_2^{q/p}\cdots \, \boldsymbol{\omega}_{\frac{p-5}{2}}^{q/p} \,\, \boldsymbol{\omega}_{\frac{p-3}{2}}^{\left(q- \frac{(p-1)}{2}\right)/p}.\]
Notice that $\boldsymbol{b}_{pq}$ is the empty word $\epsilon$ if and only if $p=5$. By Proposition~\ref{prop:gap of words wi}, we obtain 
\begin{align*}
\boldsymbol{\omega}_0^{q/p}& =\left(1(-1)0^{p-2}\right)^{\lfloor q/p\rfloor}1(-1)0^{r-2}, 
\\[1ex]
\boldsymbol{\omega}_{1}^{q/p}&=\left(1(-1)0^{p-r-2}1(-1)0^{ r-2}\right)^{\lfloor q/p\rfloor}\boldsymbol{\omega}_1^{r/p},
\end{align*}
where
\begin{align*}
\boldsymbol{\omega}_{1}^{r/p}
      =\begin{cases}  
      1(-1)0^{r-2}, &  \text{if }\, r<(p+1)/2
      \\[1ex]
      1(-1)0^{p-r-2}1, & \text{if }\, r=(p+1)/2
      \\[1ex]
      1(-1)0^{p-r-2}1(-1)0^{2r-p-2}, &\text{if }\, r>(p+1)/2.
        \end{cases}
\end{align*}
The cases $r=2$ and $r=p-2$ are covered by the convention $0^{0}=\epsilon$. 

Note that $u_{+}=2$ if and only if $r=(p+1)/2$. In this case, $\fin{\boldsymbol{\omega}_{1}^{r/p}}=r-1$. If $u_{+}\ge 3$, then  $\com{\boldsymbol{b}_{pq}}=\com{\boldsymbol{\omega}_2}=0$ (see Table~\ref{table:values of omegai}). In both cases,
\[G(\boldsymbol{c}_{pq})=G(\boldsymbol{\omega}_{0}^{q/p})\cup G(\boldsymbol{\omega}_{1}^{q/p})\cup G(\boldsymbol{b}_{pq}),\]
 since for any  $\boldsymbol{u}, \boldsymbol{v} \in\mathcal{A}^{+}$ satisfying either  
$\fin{\boldsymbol{u}}=|\boldsymbol{u}|-1$ or $\com{\boldsymbol{v}}=0$, we have $G(\boldsymbol{u} \boldsymbol{v})=G(\boldsymbol{u})\cup G(\boldsymbol{v})$.

The word $\boldsymbol{\omega}_0^{q/p}$ contributes to $G(\boldsymbol{c}_{pq})$ with $\lfloor q/p\rfloor$ gaps of length $p-1$. If $r>2$, it also provides $1$ gap of length $r-1$. 
  
The contribution of $\boldsymbol{\omega}_{1}^{q/p}$ depends on whether $r\geq {(p+1)}/{2}$ or $r< {(p+1)}/{2}$. If $r<{(p+1)}/{2}$, then $\boldsymbol{\omega}_{1}^{q/p}$ contributes  with $\lfloor q/p\rfloor$ gaps of length $p-r-1$ and  $\lfloor q/p\rfloor +1$ gaps of length $r-1$.  
If $r\ge {(p+1)}/{2}$, then there are $\lfloor q/p\rfloor$ gaps of length $r-1$, $\lfloor q/p\rfloor +1$ gaps of length $p-r-1$, and $1$ gap of length $2r-p-1$.  

Gaps of length $1$ are irrelevant. Thus, when $r=2$, gaps of length $r-1$ are not considered. The same applies when $r=p-2$, where gaps of length $p-r-1$ are ignored.

From Proposition~\ref{prop:hasta p-3/2} we deduce that the length of the gaps of $\boldsymbol{\omega}_1^{q/p}$ and of $\boldsymbol{b}_{pq}$ is upper bounded by $\max\{p-r-1, r-1\}$. In summary, we conclude:
\begin{itemize}
    \item Only $\boldsymbol{\omega}_{0}^{q/p}$ contributes to the maximum gap of $\boldsymbol{c}_{pq}$, which is equal to $p-1$ and it contributes with $\lfloor q/p\rfloor$ gaps of this length.

    \item The second gap of $\boldsymbol{c}_{pq}$ is equal to $\max\{p-r-1, r-1\}$. We can find at least $\lfloor q/p\rfloor$ such gaps in $\boldsymbol{\omega}_{1}^{q/p}$. 
\end{itemize}

To complete the proof of Theorem~\ref{thm:second gap} and conclude about the number and the positions of the maximum and second gaps, we first observe that 
\begin{align*}
  \boldsymbol{\omega}_{0}^{q/p} = \boldsymbol{c}_{pq}[0:q-1] = \boldsymbol{a}_{pq}[0:q-1],
\\
   \boldsymbol{\omega}_{1}^{q/p} = \boldsymbol{c}_{pq}[q:2q-1] = \boldsymbol{a}_{pq}[q:2q-1].
\end{align*}
Since  $\boldsymbol{a}_{pq}$ is  reciprocal and $G(\boldsymbol{a}_{pq}) = G(\boldsymbol{c}_{pq})$, the gaps from $\boldsymbol{\omega}_{0}^{q/p}$ and
$\boldsymbol{\omega}_{1}^{q/p}$ appear symmetrically in $\boldsymbol{a}_{pq}$. This symmetry determines both the number and the positions of the maximum and second gaps. This proves Theorem~\ref{thm:second gap} and thus, Theorem~\ref{thm:teorema 1}.

\section{Gaps and a representation theorem for \texorpdfstring{$\Phi_{pq}$}{} when \texorpdfstring{$q \equiv \pm 1\!\!\!\mod \!p$}{}  } \label{section:gaps and a representation theorem}

In Section \ref{sec: rp}, we prove  a representation theorem for $\boldsymbol{a}_{pq}$ when $q\equiv \pm 1 \!\!\!\mod \!p $, stated in Theorem~\ref{thm:descripcion word pm1}. This representation uses only integer powers of words whose last letter is nonzero, making it well-suited to study gaps.
 
For $p=3$, we recover a historical result due to Habermehl, Richardson, and Szwajkos~\cite{Habermehl64} (see Corollary~\ref{coro: HRS}). In Section~\ref{subsec:completing proof of theorem 2}, we complete the proof of Theorem~\ref{thm:teorema 2}.

\subsection{The representation theorem}\label{sec: rp}
\begin{theorem} \label{thm:descripcion word pm1}
The following holds for the cyclotomic word $\boldsymbol{a}_{pq} $.
\begin{enumerate}
    \item If  $ q = 1 \!\!\!\mod \!p$ and $\boldsymbol{v}_i=0^i(-1)0^{p-i-2}1$  for each $i\in \llbracket 0,p-2\rrbracket$, then
\[\boldsymbol{a}_{pq}=1\boldsymbol{v}_0^{\lfloor q/p\rfloor} \boldsymbol{v}_1^{\lfloor q/p\rfloor}\boldsymbol{v}_2^{\lfloor q/p\rfloor}\cdots \boldsymbol{v}_{p-3}^{\lfloor q/p\rfloor} \boldsymbol{v}_{p-2}^{\lfloor q/p\rfloor}.\ \]
 
\item If $ q=p-1\!\!\!\mod \!p$ and $\boldsymbol{v}_i=0^{p-i-2}10^i(-1)$ for each $i\in \llbracket 0,p-2\rrbracket$, then
\[\boldsymbol{a}_{pq}=1(-1)\boldsymbol{v}_0^{\lfloor q/p\rfloor} \boldsymbol{v}_1^{\lfloor q/p\rfloor+1}\boldsymbol{v}_2^{\lfloor q/p\rfloor+1}\cdots \boldsymbol{v}_{p-3}^{\lfloor q/p\rfloor+1} \boldsymbol{v}_{p-2}^{\lfloor q/p\rfloor}1. \]
\end{enumerate}

\end{theorem}

\begin{proof} 
The proof relies on the expression of $\boldsymbol{a}_{pq}$ given in Theorem~\ref{thm:representation theorem}.

The case 
$p=3$ is immediate from the expressions of $\boldsymbol{\omega}_{0}$ and $\boldsymbol{\omega}_{1}$. 
For $p\ge 5$, we first compute the basic words $\boldsymbol{\omega}_i$ for $i\in \llbracket 0,p-2\rrbracket$; then, we compute $\boldsymbol{\omega}_i^{q/p}$ and their concatenation; finally, an inductive argument completes the proof.

\smallskip

\noindent \emph{Case $q=1\!\!\!\mod \!p$.} Since $u_{+} = 1$ and $u_{-} = p-1$, it follows from \eqref{eq:formula for I} that 
\begin{align*}
{\sf I}_{+}(j) = \begin{cases}
\quad 0, & \text{if }\, j = 0
\\[1ex]
p-j, & \text{if }\, j \in \llbracket 1, p-1 \rrbracket 
\end{cases} \quad \mbox{and} \quad 
{\sf I}_{-}(j) = 
\begin{cases}
\phantom{p-\,} 1, & \text{if }\, j = 0
\\[1ex]
\phantom{p-\,} 0, & \text{if }\, j = 1
\\[1ex]
p-j+1, & \text{if }\, j \in \llbracket 2, p-1 \rrbracket. 
\end{cases} 
\end{align*}
From Lemma~\ref{lemma:wieI}, we have 
\[
\boldsymbol{\omega}_{i}[0]=0 \ \text{for } i \in \llbracket 1, p-2\rrbracket \quad \text{and}\quad \boldsymbol{\omega}_i[1]=-1 \ \text{for}\ i \in \llbracket 0, p-2\rrbracket, 
\]
and 
\[\boldsymbol{\omega}_{i}[j] = 1 \quad \text{ if and only if }\quad p-j={\sf I}_{+}(j)\le i<{\sf I}_{-}(j)=p-j+1.\]
 Thus, $\boldsymbol{\omega}_{i}=0(-1)0^{p-i-2}10^{i-1}$ for each $i\in \llbracket 1, p-2\rrbracket$.
 
 We use the identity $\left(\boldsymbol{u}\boldsymbol{v}\right)^n=\boldsymbol{u}\left(\boldsymbol{v}\boldsymbol{u}\right
 )^{n-1} \boldsymbol{v}$, with $\boldsymbol{u}=(-1)0^{p-i-2}1$ and $\boldsymbol{v}=0^{i-1}$, and the equality $\boldsymbol{\omega}_i^{1/p}=0$ to obtain that 
\[\boldsymbol{\omega}_i^{q/p}=\left(0(-1)0^{p-i-2}10^{i-1}\right)^{\lfloor q/p\rfloor}0=0(-1)0^{p-i-2}1\left(0^{i}(-1)0^{p-i-2}1\right)^{\lfloor q/p\rfloor-1}0^i.\]

If $ \boldsymbol{v}_i=0^i(-1)0^{p-i-2}1$ for $ i\in \llbracket 1,p-3\rrbracket$, then
 \begin{align*}
\boldsymbol{\omega}_{i}^{q/p}\boldsymbol{\omega}_{i+1}^{q/p}=0(-1)0^{p-i-2}1\boldsymbol{v}_i^{\lfloor q/p\rfloor-1}\boldsymbol{v}_{i+1}^{\lfloor q/p\rfloor}0^{i+1}.
\end{align*} 
By induction on $i$, using the explicit forms
 $\boldsymbol{v}_{1}=0(-1)0^{p-3}1$ and $\boldsymbol{\omega}_{0}^{q/p}=1\boldsymbol{v}_0^{\lfloor q/p\rfloor}$, we obtain 
\[ \boldsymbol{\omega}_{0}^{q/p}\boldsymbol{\omega}_{1}^{q/p}\cdots \boldsymbol{\omega}_{p-2}^{q/p}=1\boldsymbol{v}_{0}^{\lfloor q/p\rfloor}\boldsymbol{v}_{1}^{\lfloor q/p\rfloor}\cdots \boldsymbol{v}_{p-2}^{\lfloor q/p\rfloor}0^{p-2}.\]
Since $\boldsymbol{a}_{pq}$ is the prefix of length $\varphi(pq)=(p-1)(q-1) +1$ of the left-hand side of the previous identity, our expression is obtained by removing the suffix $0^{p-2}$ from the right-hand side.

\smallskip

\noindent \emph{Case $ q=p-1\!\!\!\mod \!p$.}
Since $u_{+}=p-1$ and $u_{-} = 1$, it follows from \eqref{eq:formula for I} that 
\begin{align*}
 {\sf I}_{+}(j) = \begin{cases}
0, & \text{if }\, j = 0
\\[1ex]
j, & \text{if }\, j \in \llbracket 1, p-1 \rrbracket
\end{cases}\quad \quad
& \mbox{ and } \quad \quad
{\sf I}_{-}(j) = \begin{cases}
p-1, & \text{if }\, j = 0
\\[1ex]
j-1, &\text{if }\, j \in \llbracket 1, p-1 \rrbracket.
\end{cases} 
\end{align*}
From Lemma~\ref{lemma:wieI} it follows immediately that
$\boldsymbol{\omega}_i[0]=1 \ \text{for} \ i \in \llbracket 0, p-2\rrbracket$
and that 
\[\boldsymbol{\omega}_{i}[j] = -1\quad \text{if and only if}\quad j-1={\sf I}_{-}(j)\le i <{\sf I}_{+}(j)=j.\]
Thus, $\boldsymbol{\omega}_i=10^{i}(-1)0^{p-i-2}$ for each $i\in \llbracket 0, p-2\rrbracket$.
 
 The identity $\left(\boldsymbol{u}\boldsymbol{v}\right)^n\boldsymbol{u}= \boldsymbol{u}\left(\boldsymbol{v}\boldsymbol{u}\right)^n$,
   with $\boldsymbol{u}=10^i(-1)$ and $\boldsymbol{v}={{0^{p-i-2}}}$, and the equality $\boldsymbol{\omega}_i^{(p-1)/p}= 10^{i}(-1)0^{p-i-3}$ give
\[\boldsymbol{\omega}_{i}^{q/p}=\left(10^i(-1)0^{p-i-2}\right)^{\lfloor q/p\rfloor}10^i(-1)0^{p-i-3} =10^i(-1)\boldsymbol{v}_{i}^{\lfloor q/p\rfloor}0^{p-i-3},\quad i\neq p-2. \]
For $0\le i\le p-4$, we obtain 
\[
\boldsymbol{\omega}_{i}^{q/p}\boldsymbol{\omega}_{i+1}^{q/p} =10^i(-1)\boldsymbol{v}_{i}^{\lfloor q/p\rfloor}\boldsymbol{v}_{i+1}^{\lfloor q/p\rfloor+1}0^{p-i-4}.
\]
If $i=p-2$, then the following identities hold 
\[\boldsymbol{\omega}_{p-2}=\boldsymbol{v}_{p-2}=10^{p-2}(-1)\quad \text{and}\quad \boldsymbol{\omega}_{p-2}^{(p-1)/p}=10^{p-2}.\]
An inductive argument and the previous equality prove that
 \[\boldsymbol{\omega}_{0}^{q/p}\boldsymbol{\omega}_{1}^{q/p}\cdots \boldsymbol{\omega}_{p-2}^{q/p}=1(-1)\boldsymbol{v}_{0}^{\lfloor q/p\rfloor}\boldsymbol{v}_{1}^{\lfloor q/p\rfloor+1}\cdots \boldsymbol{v}_{p-3}^{\lfloor q/p\rfloor+ 1}\boldsymbol{v}_{p-2}^{\lfloor q/p\rfloor}10^{p-2}.\]

To complete the proof of the second assertion, we proceed as in the first one by considering the prefix of length $(p-1)(q-1) + 1$.
\end{proof}

 The following is a historical result due to Habermehl, Richardson, and Szwajkos~\cite{Habermehl64}. 
\begin{corollary}\label{coro: HRS} Let $p=3$ and let $\boldsymbol{v}_{0}$ and $\boldsymbol{v}_{1}$ be the words defined in Theorem~\ref{thm:descripcion word pm1}. \begin{enumerate}
\item If $r=1$, then $\boldsymbol{a}_{3 q} = 1\boldsymbol{v}_{0}^{(q-1)/3} \boldsymbol{v}_{1}^{(q-1)/3}$  with $\boldsymbol{v}_{0} = (-1)01$ and $\boldsymbol{v}_{1} = 0(-1)1$. 
\smallskip
\item If $r=2$, then $\boldsymbol{a}_{3 q} = 1(-1)\boldsymbol{v}_{0}^{(q-2)/3} \boldsymbol{v}_{1}^{(q-2)/3} 1$ with $\boldsymbol{v}_{0} = 01(-1)$  and  $\boldsymbol{v}_{1} = 10(-1)$. 
\end{enumerate}
\end{corollary}

\subsection{Proof of Theorem \ref{thm:teorema 2}} \label{subsec:completing proof of theorem 2} 
If $p=3$, the result follows from Corollary \ref{coro: HRS}. 
Now assume $p\ge 5$. Recall that a gapblock $0^{k}$ inside a word produces a gap of length $k+1$ (see Section~\ref{sect: results-notations}).
By Theorem~\ref{thm:descripcion word pm1}, in both cases $r=1$ and $r=p-1$, the words $\boldsymbol{v}_{0}$ and $\boldsymbol{v}_{p-2}$ each produce a gap of length $p-1$. Moreover, for $i \in \llbracket 1, p-3\rrbracket$, each of $\boldsymbol{v}_{i}$ and $\boldsymbol{v}_{p-i-2}$ provides gaps of length $i+1$ and $p-i-1$, respectively. The crucial fact here is that in any case, for $i \in \llbracket 0,p-2 \rrbracket$ the last letter of $\boldsymbol{v}_{i}$ is either $1$ or $-1$. Thus, for $i\in \llbracket 2,p-1\rrbracket$, the number of gaps of length $i$ of $\boldsymbol{a}_{pq}$ is twice the number of gaps of length $i$ of $\boldsymbol{v}_{i-1}^{\lfloor q/p\rfloor}$.

Hence, if $r=1$, there are $2\lfloor q/p\rfloor$ gaps of length $i$ for $i \in \llbracket 2, p-1\rrbracket$. If $r=p-1$, then there are $2\lfloor q/p\rfloor$ gaps of length $p-1$ and $2\lfloor q/p\rfloor + 2$ gaps of length $i$ for $i \in \llbracket 2, p-2\rrbracket$.

\subsubsection*{Acknowledgments}
We thank the anonymous referee for the careful reading of the manuscript and for the valuable suggestions that helped improve the overall presentation of the paper. We are grateful to Brigitte Vall\'ee and Alfredo Viola for many encouraging discussions on this subject.

\noindent This work was partially supported by Projects UNGS 30/3373 and 30/3429, and Project STIC AmSud 20STIC-06 EPAA.

\end{document}